\documentclass[12pt]{amsart}
\usepackage{amsmath,amssymb,amsthm,graphics,amscd,mathrsfs}
\usepackage{amscd, amsfonts, amssymb, amsthm}
\usepackage{fullpage, verbatim}
\usepackage[colorlinks, breaklinks, linkcolor=blue]{hyperref} 
\usepackage{breakurl}
\usepackage{array}
\usepackage{latexsym}
\usepackage{enumerate}
\usepackage{pgf}

\newcommand\CA{{\mathscr A}} 
\newcommand\CB{{\mathscr B}}
 
\renewcommand\CD{{\mathscr D}}

\newcommand\CI{{\mathcal I}}

\newcommand\BBC{{\mathbb C}}

\newcommand\BBR{{\mathbb R}}
\newcommand\BBZ{{\mathbb Z}}

\newcommand\codim{\operatorname{codim}}

\numberwithin{equation}{section}

\theoremstyle{plain}
\newtheorem{lemma}[equation]{Lemma}
\newtheorem{theorem}[equation]{Theorem}

\theoremstyle{definition}

\newtheorem{remark}[equation]{Remark}

\newtheorem{example}[equation]{Example}


\begin{document}

\title[Restrictions of aspherical arrangements]
{Restrictions of aspherical arrangements}

\author[N. Amend]{Nils Amend}
\address
{Institut f\"ur Algebra,~Zahlentheorie und Diskrete Mathematik,
Fakult\"at f\"ur Mathematik und Physik,
Gottfried Wilhelm Leibniz Universit\"at Hannover,
Welfengarten 1, D-30167 Hannover, Germany}
\email{amend@math.uni-hannover.de}

\author[T. M\"oller]{Tilman M\"oller}
\address
{Fakult\"at f\"ur Mathematik,
Ruhr-Universit\"at Bochum,
D-44780 Bochum, Germany}
\email{tilman.moeller@rub.de}

\author[G. R\"ohrle]{Gerhard R\"ohrle}
\address
{Fakult\"at f\"ur Mathematik,
Ruhr-Universit\"at Bochum,
D-44780 Bochum, Germany}
\email{gerhard.roehrle@rub.de}


\allowdisplaybreaks

\begin{abstract}
In this note we present examples 
of $K(\pi,1)$-arrangements which 
admit a restriction which fails to be $K(\pi,1)$. 
This shows that asphericity is not hereditary among
hyperplane arrangements.
\end{abstract}

\maketitle


\section{Introduction}

We say that a complex $\ell$-arrangement $\CA$ is a 
\emph{$K(\pi,1)$-arrangement}, or that $\CA$ is $K(\pi,1)$ for short, provided 
the complement $M(\CA)$ of the union of the hyperplanes in 
$\CA$ in $\BBC^\ell$ is aspherical, i.e.~is a 
$K(\pi,1)$-space. That is, the universal covering space of $M(\CA)$ 
is contractible and the fundamental group
$\pi_1(M(\CA))$ of $M(\CA)$ is isomorphic to the group $\pi$.

By seminal work of Deligne \cite{deligne}
(the complexification of)
a \emph{simplicial} real arrangement is 
$K(\pi,1)$.

Another important class of arrangements which 
share this topological property are arrangements 
of \emph{fiber type} \cite{falkrandell:fiber-type}.
For such an arrangement $\CA$ the complement  $M(\CA)$ 
is described as a tower of successive
locally trivial linear fibrations with 
aspherical fibers and aspherical bases.
A repeated application of the
long exact sequence in homotopy theory
then gives that such $\CA$ are $K(\pi,1)$
(cf.~\cite[Prop.~5.12]{orlikterao:arrangements}).
By fundamental work of Terao \cite{terao:modular}, this 
property in turn is equivalent to \emph{supersolvability}
of $\CA$
(cf.~\cite[Thm.~5.113]{orlikterao:arrangements}).

A particularly important and prominent class of arrangements
for which this property is known to hold is the class of 
reflection arrangements stemming from 
complex reflection groups. 
In 1962,  Fadell and Neuwirth \cite{fadellneuwirth}
proved that the complexified braid arrangement 
of the symmetric group 
is  $K(\pi,1)$. 
Brieskorn \cite{brieskorn:tresses} extended this 
result to many finite
Coxeter groups and conjectured that this is the case for 
every finite Coxeter group.
As the latter are simplicial, 
this follows from Deligne's seminal work \cite{deligne}.
Nakamura proved asphericity for the imprimitive complex reflection groups, 
constructing explicit locally trivial fibrations \cite{nakamura}.
Utilizing their approach 
via \emph{Shephard groups}, Orlik and Solomon 
succeeded in showing that the reflection arrangements 
stemming from the remaining irreducible
complex reflection groups admit complements which 
are $K(\pi,1)$-spaces with the possible exception of 
just six cases \cite{orliksolomon:discriminant}; 
see also \cite[\S 6]{orlikterao:arrangements}.
These remaining instances were settled by Bessis in 
his brilliant proof employing Garside categories \cite{bessis:kpione}.

Because restrictions of simplicial (resp.~supersolvable)
arrangements are again simplicial 
(resp.~supersolvable), 
the $K(\pi, 1)$-property of these kinds of arrangements
is inherited by their restrictions.

In contrast, as the restriction of a reflection arrangement 
need not be a reflection arrangement again,
the question whether asphericity passes to 
restrictions of reflection arrangements is more subtle.
Nevertheless, it turns out that restrictions
of reflection arrangements  
are indeed again $K(\pi,1)$ with only 11 possible exceptions,
see \cite{amendroehrle:kpione}. Conjecturally, these remaining
instances  are also $K(\pi,1)$.

Nevertheless, the impression of a 
hereditary behavior of asphericity
alluded to by these particular classes is deceptive.
In this note we present examples of 
$K(\pi,1)$-arrangements which admit non-$K(\pi,1)$ restrictions,
giving the following.

\begin{theorem}
\label{thm:main}
The class of $K(\pi,1)$-arrangements is not closed under
taking restrictions.
\end{theorem}

Specifically, we construct infinite families  
of $K(\pi,1)$-subarrangements of 
Coxeter arrangements of type $D_n$ for any $n \ge 4$
each of which 
admits a restriction that fails to
be $K(\pi,1)$, see Lemma \ref{lem:dn} and Example \ref{ex:dn}. 
Our smallest example 
of this kind is a rank 4 subarrangement of the Coxeter arrangement of
type $D_4$ consisting of just 10 hyperplanes, see Example \ref{ex:d4}.
According to our knowledge, this is the first instance 
of the description of such examples in the literature.
Likely, this is not a particularly rare phenomenon.

It is remarkable that this phenomenon 
appears naturally among canonical subarrangements
of reflection arrangements.
This shows quite dramatically, while 
Coxeter arrangements themselves 
are  well understood, 
their subarrangements 
still hold some unexpected surprises. 

In contrast to the situation with restrictions, 
any localization of a 
$K(\pi,1)$-arrangement
is known to be $K(\pi,1)$ again, by an observation due to
Oka, e.g., see \cite{paris:deligne}.

\section{Preliminaries}
\label{sect:prelims}

\subsection{Hyperplane arrangements}
\label{ssect:arrangements}
Let $V = \BBC^\ell$ be an $\ell$-dimensional $\BBC$-vector space.
A \emph{hyperplane arrangement} $\CA = (\CA, V)$ in $V$ 
is a finite collection of hyperplanes in $V$.
We also use the term $\ell$-arrangement for $\CA$ to 
indicate the dimension of the ambient space $V$. 
If the linear equations describing all members of
$\CA$ are real, then we say that $\CA$ is  
\emph{real}.

The \emph{lattice} $L(\CA)$ of $\CA$ is the set of subspaces of $V$ of
the form $H_1\cap \ldots \cap H_i$ where $\{ H_1, \ldots, H_i\}$ is a subset
of $\CA$. 
The lattice $L(\CA)$ is a partially ordered set by reverse inclusion:
$X \le Y$ provided $Y \subseteq X$ for $X,Y \in L(\CA)$.

For $X \in L(\CA)$, we have two associated arrangements, 
firstly
$\CA_X :=\{H \in \CA \mid X \subseteq H\} \subseteq \CA$,
the \emph{localization of $\CA$ at $X$}, 
and secondly, 
the \emph{restriction of $\CA$ to $X$}, $(\CA^X,X)$, where 
$\CA^X := \{ X \cap H \mid H \in \CA \setminus \CA_X\}$.

If $0 \in H$ for each $H$ in $\CA$, then 
$\CA$ is called \emph{central}.
If $\CA$ is central, then the \emph{center} 
$\cap_{H \in \CA} H$ of $\CA$ is the unique
maximal element in $L(\CA)$  with respect
to the partial order.
We have a \emph{rank} function on $L(\CA)$: $r(X) := \codim_V(X)$.
The \emph{rank} $r := r(\CA)$ of $\CA$ 
is the rank of a maximal element in $L(\CA)$.
Throughout this article, we only consider central arrangements.
For $\CA$ central, for each $H$ in $\CA$ let $\alpha_H$ be a 
linear form in $V^*$ so that $H = \ker \alpha_H$.
Then $Q(\CA) = \prod_{H\in \CA}\alpha_H$ is the 
\emph{defining polynomial} of $\CA$.

\bigskip
It is easy to see that a central 
arrangement of rank at most $2$ is $K(\pi, 1)$
(\cite[Prop.~5.6]{orlikterao:arrangements}).

This topological property is not generic among all arrangements.
A \emph{generic} complex 
$\ell$-arrangement $\CA$ is an $\ell$-arrangement with at least $\ell + 1$ hyperplanes and
the property that the hyperplanes of every subarrangement
$\CB \subseteq \CA$ with $\vert\CB\vert = \ell$
are linearly independent. It follows from work of Hattori \cite{hattori} that, for $\ell\geq 3$,
generic arrangements are never $K(\pi, 1)$ (cf.~\cite[Cor.~5.23]{orlikterao:arrangements}).

\begin{lemma}
\label{lem:nonkpione}
The real 3-arrangement $\CA$
given by 
$Q(\CA) = y (x-y) (x^2-z^2) (y^2-z^2)$ is not $K(\pi,1)$.
\end{lemma}

\begin{proof}
For $t \in \BBC$, let  
$\CA(t)$ be the 3-arrangement
given by 
$Q(\CA(t)) = x y z (x+y) (x+z) (y+tz)$. 
Using the transformation 
$(x,y,z)\mapsto (x-y,y-z,y+z)$, 
it is easy to see that 
$\CA$ and $\CA(1)$ are linearly isomorphic, 
so their complements are diffeomorphic.
For $t \in \BBC \setminus \{0,-1\}$ 
the arrangements $\CA(t)$ are all 
combinatorially isomorphic. 
Moreover, for real negative $t$ they satisfy 
the ``simple triangle'' condition of Falk and Randell, 
cf.~\cite[Cor.~3.3, (3.12)]{falkrandell:homotopy}, so $\CA(t)$ is not $K(\pi,1)$ for $t<0$.

For completeness we show 
that $\CA(1)$ and $\CA(-2)$ 
have diffeomorphic complements, which shows that 
$\CA(1)\cong \CA$ is not $K(\pi,1)$. 
For $t\in\BBC$ define the 1-parameter family 
$ \CB_t=\CA(\frac{3}{2}e^{\pi i t}-\frac{1}{2})$. Then for every $t\in \BBC$, 
$\CB_t$ admits the same lattice, 
so  $\CB_t$ is a \emph{lattice isotopy}, see \cite{randell:lattice-isotopy}. 
It follows from Randell's isotopy theorem \cite{randell:lattice-isotopy} that $\CB_0 = \CA(1)$ 
and $\CB_1 = \CA(-2)$ have diffeomorphic complements. In Figure \ref{fig:aminustwo} we show a projective picture of the real arrangement $\CA(-2)$ with the simple triangle shaded in gray.

Note that the arrangement $X_3$ considered by Falk and Randell in 
\cite[(2.6)]{falkrandell:homotopy} has the same lattice as $\CA$.
\end{proof}

\begin{figure}[h]
\resizebox{0.46\textwidth}{!}{
\begingroup%
\makeatletter%
\begin{pgfpicture}%
\pgfpathrectangle{\pgfpointorigin}{\pgfqpoint{5.840000in}{5.840000in}}%
\pgfusepath{use as bounding box, clip}%
\begin{pgfscope}%
\pgfsetbuttcap%
\pgfsetmiterjoin%
\definecolor{currentfill}{rgb}{1.000000,1.000000,1.000000}%
\pgfsetfillcolor{currentfill}%
\pgfsetlinewidth{0.000000pt}%
\definecolor{currentstroke}{rgb}{1.000000,1.000000,1.000000}%
\pgfsetstrokecolor{currentstroke}%
\pgfsetdash{}{0pt}%
\pgfpathmoveto{\pgfqpoint{0.000000in}{0.000000in}}%
\pgfpathlineto{\pgfqpoint{5.840000in}{0.000000in}}%
\pgfpathlineto{\pgfqpoint{5.840000in}{5.840000in}}%
\pgfpathlineto{\pgfqpoint{0.000000in}{5.840000in}}%
\pgfpathclose%
\pgfusepath{fill}%
\end{pgfscope}%
\begin{pgfscope}%
\pgfsetbuttcap%
\pgfsetmiterjoin%
\definecolor{currentfill}{rgb}{1.000000,1.000000,1.000000}%
\pgfsetfillcolor{currentfill}%
\pgfsetlinewidth{0.000000pt}%
\definecolor{currentstroke}{rgb}{0.000000,0.000000,0.000000}%
\pgfsetstrokecolor{currentstroke}%
\pgfsetstrokeopacity{0.000000}%
\pgfsetdash{}{0pt}%
\pgfpathmoveto{\pgfqpoint{0.100000in}{0.100000in}}%
\pgfpathlineto{\pgfqpoint{5.740000in}{0.100000in}}%
\pgfpathlineto{\pgfqpoint{5.740000in}{5.740000in}}%
\pgfpathlineto{\pgfqpoint{0.100000in}{5.740000in}}%
\pgfpathclose%
\pgfusepath{fill}%
\end{pgfscope}%
\begin{pgfscope}%
\pgfpathrectangle{\pgfqpoint{0.100000in}{0.100000in}}{\pgfqpoint{5.640000in}{5.640000in}} %
\pgfusepath{clip}%
\pgfsetbuttcap%
\pgfsetmiterjoin%
\definecolor{currentfill}{rgb}{0.827451,0.827451,0.827451}%
\pgfsetfillcolor{currentfill}%
\pgfsetlinewidth{0.000000pt}%
\definecolor{currentstroke}{rgb}{0.827451,0.827451,0.827451}%
\pgfsetstrokecolor{currentstroke}%
\pgfsetdash{}{0pt}%
\pgfpathmoveto{\pgfqpoint{1.112308in}{4.727692in}}%
\pgfpathlineto{\pgfqpoint{2.016154in}{4.727692in}}%
\pgfpathlineto{\pgfqpoint{2.016154in}{3.823846in}}%
\pgfpathclose%
\pgfusepath{fill}%
\end{pgfscope}%
\begin{pgfscope}%
\pgfpathrectangle{\pgfqpoint{0.100000in}{0.100000in}}{\pgfqpoint{5.640000in}{5.640000in}} %
\pgfusepath{clip}%
\pgfsetrectcap%
\pgfsetroundjoin%
\pgfsetlinewidth{1.003750pt}%
\definecolor{currentstroke}{rgb}{0.000000,0.000000,0.000000}%
\pgfsetstrokecolor{currentstroke}%
\pgfsetdash{}{0pt}%
\pgfpathmoveto{\pgfqpoint{0.208462in}{4.727692in}}%
\pgfpathlineto{\pgfqpoint{5.631538in}{4.727692in}}%
\pgfpathlineto{\pgfqpoint{5.631538in}{4.727692in}}%
\pgfusepath{stroke}%
\end{pgfscope}%
\begin{pgfscope}%
\pgfpathrectangle{\pgfqpoint{0.100000in}{0.100000in}}{\pgfqpoint{5.640000in}{5.640000in}} %
\pgfusepath{clip}%
\pgfsetrectcap%
\pgfsetroundjoin%
\pgfsetlinewidth{1.003750pt}%
\definecolor{currentstroke}{rgb}{0.000000,0.000000,0.000000}%
\pgfsetstrokecolor{currentstroke}%
\pgfsetdash{}{0pt}%
\pgfpathmoveto{\pgfqpoint{0.208462in}{2.920000in}}%
\pgfpathlineto{\pgfqpoint{5.631538in}{2.920000in}}%
\pgfpathlineto{\pgfqpoint{5.631538in}{2.920000in}}%
\pgfusepath{stroke}%
\end{pgfscope}%
\begin{pgfscope}%
\pgfpathrectangle{\pgfqpoint{0.100000in}{0.100000in}}{\pgfqpoint{5.640000in}{5.640000in}} %
\pgfusepath{clip}%
\pgfsetrectcap%
\pgfsetroundjoin%
\pgfsetlinewidth{1.003750pt}%
\definecolor{currentstroke}{rgb}{0.000000,0.000000,0.000000}%
\pgfsetstrokecolor{currentstroke}%
\pgfsetdash{}{0pt}%
\pgfpathmoveto{\pgfqpoint{2.920000in}{0.208462in}}%
\pgfpathlineto{\pgfqpoint{2.920000in}{5.631538in}}%
\pgfpathlineto{\pgfqpoint{2.920000in}{5.631538in}}%
\pgfusepath{stroke}%
\end{pgfscope}%
\begin{pgfscope}%
\pgfpathrectangle{\pgfqpoint{0.100000in}{0.100000in}}{\pgfqpoint{5.640000in}{5.640000in}} %
\pgfusepath{clip}%
\pgfsetrectcap%
\pgfsetroundjoin%
\pgfsetlinewidth{1.003750pt}%
\definecolor{currentstroke}{rgb}{0.000000,0.000000,0.000000}%
\pgfsetstrokecolor{currentstroke}%
\pgfsetdash{}{0pt}%
\pgfpathmoveto{\pgfqpoint{2.016154in}{0.208462in}}%
\pgfpathlineto{\pgfqpoint{2.016154in}{5.631538in}}%
\pgfpathlineto{\pgfqpoint{2.016154in}{5.631538in}}%
\pgfusepath{stroke}%
\end{pgfscope}%
\begin{pgfscope}%
\pgfpathrectangle{\pgfqpoint{0.100000in}{0.100000in}}{\pgfqpoint{5.640000in}{5.640000in}} %
\pgfusepath{clip}%
\pgfsetrectcap%
\pgfsetroundjoin%
\pgfsetlinewidth{1.003750pt}%
\definecolor{currentstroke}{rgb}{0.000000,0.000000,0.000000}%
\pgfsetstrokecolor{currentstroke}%
\pgfsetdash{}{0pt}%
\pgfpathmoveto{\pgfqpoint{0.208462in}{5.631538in}}%
\pgfpathlineto{\pgfqpoint{5.631538in}{0.208462in}}%
\pgfpathlineto{\pgfqpoint{5.631538in}{0.208462in}}%
\pgfusepath{stroke}%
\end{pgfscope}%
\begin{pgfscope}%
\pgfpathrectangle{\pgfqpoint{0.100000in}{0.100000in}}{\pgfqpoint{5.640000in}{5.640000in}} %
\pgfusepath{clip}%
\pgfsetbuttcap%
\pgfsetmiterjoin%
\pgfsetlinewidth{0.501875pt}%
\definecolor{currentstroke}{rgb}{0.000000,0.000000,0.000000}%
\pgfsetstrokecolor{currentstroke}%
\pgfsetdash{}{0pt}%
\pgfpathmoveto{\pgfqpoint{2.920000in}{-0.062692in}}%
\pgfpathcurveto{\pgfqpoint{3.711019in}{-0.062692in}}{\pgfqpoint{4.469747in}{0.251583in}}{\pgfqpoint{5.029082in}{0.810918in}}%
\pgfpathcurveto{\pgfqpoint{5.588417in}{1.370253in}}{\pgfqpoint{5.902692in}{2.128981in}}{\pgfqpoint{5.902692in}{2.920000in}}%
\pgfpathcurveto{\pgfqpoint{5.902692in}{3.711019in}}{\pgfqpoint{5.588417in}{4.469747in}}{\pgfqpoint{5.029082in}{5.029082in}}%
\pgfpathcurveto{\pgfqpoint{4.469747in}{5.588417in}}{\pgfqpoint{3.711019in}{5.902692in}}{\pgfqpoint{2.920000in}{5.902692in}}%
\pgfpathcurveto{\pgfqpoint{2.128981in}{5.902692in}}{\pgfqpoint{1.370253in}{5.588417in}}{\pgfqpoint{0.810918in}{5.029082in}}%
\pgfpathcurveto{\pgfqpoint{0.251583in}{4.469747in}}{\pgfqpoint{-0.062692in}{3.711019in}}{\pgfqpoint{-0.062692in}{2.920000in}}%
\pgfpathcurveto{\pgfqpoint{-0.062692in}{2.128981in}}{\pgfqpoint{0.251583in}{1.370253in}}{\pgfqpoint{0.810918in}{0.810918in}}%
\pgfpathcurveto{\pgfqpoint{1.370253in}{0.251583in}}{\pgfqpoint{2.128981in}{-0.062692in}}{\pgfqpoint{2.920000in}{-0.062692in}}%
\pgfpathclose%
\pgfusepath{stroke}%
\end{pgfscope}%
\end{pgfpicture}%
\makeatother%
\endgroup%
}
\caption{The real arrangement $\CA(-2)$.}
\label{fig:aminustwo}
\end{figure}

\subsection{Arrangements of ideal type}
\label{ssect:ai}

Our examples which imply Theorem \ref{thm:main}
 stem from a particular class of 
subarrangements of Coxeter arrangements which we 
describe very briefly. 

Let $\Phi$ be an irreducible, reduced root system
and let $\Phi^+$ be the set of positive roots 
with respect to some set of simple roots $\Pi$.
An \emph{(upper) order ideal}, 
or simply \emph{ideal} for short, 
of  $\Phi^+$, is a subset $\CI$  of $\Phi^+$ 
satisfying the following condition: 
if $\alpha \in \CI$ and $\beta \in \Phi^+$ so that 
$\alpha + \beta \in \Phi^+$, then $\alpha + \beta \in \CI$.

Recall the standard partial ordering 
$\preceq$ on $\Phi$: $\alpha \preceq \beta$
provided $\beta - \alpha$ is a $\BBZ_{\ge0}$-linear combination 
of positive roots, or $\beta = \alpha$. Then $\CI$ is an ideal in $\Phi^+$
if and only if whenever 
$\alpha \in \CI$ and $\beta \in \Phi^+$ so that 
$\alpha \preceq \beta$, then $\beta \in \CI$.
The \emph{generators} of a given ideal $\CI$
are simply the elements in $\CI$ which are minimal 
with respect to $\preceq$.

Let $\CI \subseteq \Phi^+$ be an ideal and let $\CI^c := \Phi^+ \setminus \CI$
be its complement in 
$\Phi^+$. 
Let $\CA(\Phi)$ be the \emph{Weyl arrangement} of $\Phi$,
i.e., $\CA(\Phi) = \{ H_\alpha \mid \alpha \in \Phi^+\}$,
where $H_\alpha$ is the hyperplane in the Euclidean space
$V = \BBR \otimes \BBZ \Phi$ orthogonal to the root $\alpha$.
Following \cite[\S 11]{sommerstymoczko}, 
we associate with an ideal $\CI$ in $\Phi^+$ the arrangement 
consisting of all hyperplanes with respect to the roots in $\CI^c$.
The \emph{arrangement of ideal type} associated with 
$\CI$ is the subarrangement $\CA_\CI$
of $\CA(\Phi)$ defined by 
\[
\CA_\CI := \{ H_\alpha \mid \alpha \in \CI^c\}.
\]

Note that if $\CI = \varnothing$, then 
$\CA_\CI = \CA(\Phi)$ is just the reflection arrangement of $\Phi$ and so 
$\CA_\varnothing$ is $K(\pi,1)$ by Deligne's result.
It is shown in 
\cite[Thm.~1.4]{amendroehrle:ideal} 
that all arrangements of ideal type 
$\CA_\CI$ are $K(\pi,1)$ 
provided the underlying root system is classical.
This is also known for most  $\CA_\CI$
for root systems of exceptional type,  
\cite[Thm.~1.3(iii)]{amendroehrle:ideal}
and is conjectured to hold for all $\CA_\CI$.

\section{Proof of Theorem \ref{thm:main}}

In this section we 
label the positive roots in a root system of type $D_n$ as in
\cite[Planche IV]{bourbaki:groupes}.
In view of \cite[Thm.~1.4]{amendroehrle:ideal}, 
Theorem \ref{thm:main} is a consequence of the following result.

\begin{lemma}
\label{lem:dn}
Let  $\Phi$ be the root system of 
type $D_n$ for $n\ge 4$.
We consider two kinds of ideal arrangements $\CA_\CI$ by listing the generators of 
the corresponding ideals $\CI$:
\begin{itemize}
\item[(i)] $1\ldots1\!\stackrel{1}{_{1}} \ = e_1+e_{n-1}$; 

\item[(ii)] $1\ldots1\!\stackrel{1}{_{1}} \ = e_1+e_{n-1}$, 
$0\ldots01\ldots12\ldots2\!\stackrel{1}{_{1}} \ =  e_s+e_t$, where $1 <  s < t < n-1$. 
Here $s$ is the first position with a coefficient $1$ and $t$ is the first position labeled with $2$.
\end{itemize}
Consider
$$
Y:= \bigcap_{2 \le i < j \le n-1} \ker(x_i-x_j)
$$
in $L(\CA_\CI)$.
Then the rank $3$ restriction $\CA_\CI^Y$ is not $K(\pi,1)$.
\end{lemma}

\begin{proof}
Let $\CD_n$ be the reflection arrangement of $D_n$, i.e.
$$
\CD_n = \left\{\ker(x_i\pm x_j)\ \middle|\ 1 \leq i < j \leq n\right\}.
$$
If $\CI$ is of type (i), then $\CA_\CI$ is the arrangement
$$
\CA_\CI = \CD_n\setminus\left\{\ker(x_1+x_i)\ \middle|\ 2 \le i \le n-1\right\},
$$
and if $\CI$ is of type (ii), then we have
$$
\CA_\CI = \CD_n\setminus\left\{\ker(x_1+x_i), \ker(x_j+x_k)\ \middle|\ 2 \le i \le n-1,\ 2 \le j \leq s < k \le t\right\}.
$$
In both cases, restricting to $Y$ we get the rank $3$ arrangement $\CA_\CI^Y$ with defining polynomial
$$
Q(\CA_\CI^Y) = (x_1-x_2)x_2(x_1^2-x_n^2)(x_2^2-x_n^2) \in \BBC\left[x_1,x_2,x_n\right],
$$
and thus by Lemma \ref{lem:nonkpione}, the restriction is not $K(\pi,1)$.
\end{proof}

The following example illustrates the smallest instance from Lemma \ref{lem:dn}(i).

\begin{example}
\label{ex:d4}
Let  $\CI$ be the
ideal in the set of positive 
roots in the root system of 
type $D_4$
generated by 
$\stackrel{111}{_{1}}\ = e_1+e_3$, the unique root of height $4$ and let
$\CA_\CI$ be the corresponding arrangement.
It is obtained 
from the full Weyl arrangement of 
type $D_4$ by removing the 
hyperplanes corresponding to the
two highest roots, so that
$\CA_\CI$ has defining polynomial
\[
Q(\CA_\CI) = 
(x_1-x_2)(x_1-x_3)(x_2^2-x_3^2)(x_1^2-x_4^2)(x_2^2-x_4^2)(x_3^2-x_4^2).
\]
Consider the restriction of 
$\CA_\CI$ to the hyperplane 
$H:= \ker(x_2-x_3)$. Then 
$\CA_\CI^H$ 
has defining polynomial 
$y (x-y) (x^2-z^2) (y^2-z^2)$.

Remarkably, 
this is the smallest instance among the class of 
arrangements of ideal type which is neither supersolvable nor 
simplicial. So this phenomenon of a non-$K(\pi,1)$
restriction appears among the family of 
arrangements of ideal type already for the smallest
instance where this is possible.
\end{example}

The cases from Lemma \ref{lem:dn} lead to additional instances 
of $K(\pi,1)$-arrangements with non-$K(\pi,1)$ restrictions
by means of localizations, as illustrated by our next example.

\begin{example}
\label{ex:dn}
Let $\Phi$ be of type $D_n$ and let $\CI$ be an ideal in $\Phi^+$ which admits 
the root $0\ldots01\ldots11\!\stackrel{1}{_{1}} \ = e_r+e_{n-1}$
as a generator, 
where $1 < r \le n-3$. 
Let $\CB:=\CA_\CI$ be the corresponding
arrangement of ideal type. 
Let $\Phi_0$ be the standard subsystem of type $D_m$,
where $m:= n-r+1$, and 
let $X := \cap_{\gamma \in \Phi_0^+}H_\gamma$.
Then $X$ belongs to $L(\CB)$ and the localization
$\CB_{X}$ is isomorphic to one of the arrangements of ideal type
in type $D_m$ considered in Lemma \ref{lem:dn} above.
Consequently, choosing the member $Y$ of the lattice of $\CB$ corresponding to
the one used in the proof of Lemma \ref{lem:dn},
we see that $(\CB_{X})^Y$ is isomorphic 
to the non-$K(\pi,1)$ restriction 
from Lemma \ref{lem:dn}.
Since $(\CB^Y)_{X} = (\CB_{X})^Y$, 
it follows from Oka's observation that also 
$\CB^Y$ is not $K(\pi,1)$.
However, 
$\CB$ itself is $K(\pi,1)$, thanks to 
\cite[Thm.~1.4]{amendroehrle:ideal}.
\end{example}

The following example illustrates 
all instances of an arrangement of ideal type 
$\CA_\CI$ for a root system of type $D_5$
when there is a rank $3$ restriction which admits a ``simple triangle'' and is thus not $K(\pi,1)$.
All of them stem from the constructions outlined in Lemma \ref{lem:dn} and  Example \ref{ex:dn}.
Computations for higher rank instances suggest that this is always the case.

\begin{example}
\label{ex:d5}
Let $\Phi$ be of type $D_5$ and let $\CI$ be an ideal in $\Phi^+$. Then 
$\CA_\CI$ has a rank $3$ restriction which admits a ``simple triangle'' and is thus not $K(\pi,1)$
if and only if $\CI$ is one of the following ideals (we again just list the generators of $\CI$):

\[
\begin{array}{llll}
(i) & 111\!\stackrel{1}{_{1}} & (v) & 011\!\stackrel{1}{_{1}}, 110\!\stackrel{0}{_{0}} \\
(ii) & 111\!\stackrel{1}{_{1}}, 012\!\stackrel{1}{_{1}} & (vi) & 011\!\stackrel{1}{_{1}}, 111\!\stackrel{0}{_{0}}\\
(iii) & 011\!\stackrel{1}{_{1}} & (vii) & 011\!\stackrel{1}{_{1}}, 111\!\stackrel{1}{_{0}} \\
(iv) & 011\!\stackrel{1}{_{1}}, 100\!\stackrel{0}{_{0}} & (viii) & 011\!\stackrel{1}{_{1}}, 111\!\stackrel{0}{_{1}}, 111\!\stackrel{1}{_{0}} \\
\end{array}
\]
Here the cases (i) and (ii) stem from Lemma \ref{lem:dn}(i) and (ii) 
and the cases (iii) - (viii) are covered in  Example \ref{ex:dn}.
\end{example}

While the arrangement of ideal type in our
following example stems from the one 
considered in Example \ref{ex:d4} by merely adding a single hyperplane,
the topological properties of the restrictions of both arrangements
differ sharply.

\begin{example}
\label{ex:d4-2}
Let  $\CI$ be the
ideal in the set of positive 
roots in the root system of 
type $D_4$
generated by the highest root 
$\stackrel{121}{_{1}}\ = e_1+e_2$ and let
$\CA_\CI$ be the corresponding
arrangement, i.e.~$\CA_\CI$ is obtained from the full Weyl arrangement of $D_4$ 
by removing the hyperplane corresponding to the highest root. 
One checks that $\CA_\CI$ is neither supersolvable nor simplicial.
It turns out that every restriction of $\CA_\CI^H$ to a hyperplane $H$ 
is still \emph{factored}, see \cite{terao:factored}.
It thus follows from \cite{paris:factored}
that each $\CA_\CI^H$ is still $K(\pi,1)$.
\end{example}

\begin{remark}
\label{rem:dn}
Note that while all $\CA_\CI$ are free, see \cite[Th.~11.1]{sommerstymoczko} and \cite[Thm.~1.1]{abeetall:weyl}, the
particular restrictions $\CA_\CI^Y$ considered in Lemma \ref{lem:dn} and in  Example \ref{ex:dn}
are not free, see \cite[(3.12)]{falkrandell:homotopy}.
Consequently, arrangements of ideal type are not
hereditarily free.
Nevertheless, all $\CA_\CI$ are actually inductively free, see 
\cite{roehrle:ideal} and \cite{cuntzroehrleschauenburg:ideal}.

In contrast, the ambient Weyl arrangement 
$\CA(\Phi)$ itself is hereditarily free, see \cite{orlikterao:free}. 
In  \cite{douglass:adjoint}, Douglass
gave a uniform proof
of this fact using an elegant conceptual Lie theoretic argument.
\end{remark}


\bigskip {\bf Acknowledgments}: 
The research of this work was supported by 
DFG-grant RO 1072/16-1.


\bibliographystyle{amsalpha}

\newcommand{\etalchar}[1]{$^{#1}$}
\providecommand{\bysame}{\leavevmode\hbox to3em{\hrulefill}\thinspace}
\providecommand{\MR}{\relax\ifhmode\unskip\space\fi MR }
\providecommand{\MRhref}[2]{%
  \href{http://www.ams.org/mathscinet-getitem?mr=#1}{#2} }
\providecommand{\href}[2]{#2}


\end{document}